\documentclass[reqno,12pt,letterpaper]{amsart}
\usepackage{amsmath,amssymb,amsthm,graphicx,mathrsfs,url}
\usepackage[usenames,dvipsnames]{color}
\usepackage[colorlinks=true,linkcolor=Red,citecolor=Green]{hyperref}
\usepackage{amsxtra}
\usepackage[english]{babel}
\usepackage{theoremref}

\usepackage[T1]{fontenc}
\usepackage[utf8]{inputenc}

\usepackage{tgtermes}

\def\?[#1]{\textbf{[#1]}\marginpar{\Large{\textbf{??}}}}

\let\epsilon=\varepsilon 

\newtheorem{theorem}{Theorem}
\newtheorem{proposition}{Proposition}
\newtheorem{definition}[proposition]{Definition}

\newtheorem{lemma}[proposition]{Lemma}
\newtheorem{corollary}[proposition]{Corollary}
\newtheorem{remark}{Remark}
\newtheorem{example}{Example}
\numberwithin{equation}{section}

\title{Moser's Method and Conservative Extensions of Diffeomorphisms}
\author{James Leng}
\date{February 2020}

\begin{document}

\maketitle

\begin{abstract}
This paper shall be concerned with three main results. After a brief recollection of basic symplectic geometry, we prove using Moser's \textit{homotopy method} a special case of the \textit{Strong Darboux Theorem} found, for instance, in \cite[Theorem 21.1.6]{Hor}. Next, we'll prove two conservative extension results for a diffeomorphism on a circle. One uses Moser's homotopy method but loses a degree of regularity. The other uses the method of generating functions as found in \cite{BCW} and \cite{BGV}. Finally, we'll prove a conservative extension result for a ``diffeomorphism" defined on the boundary of $(0, 1)^2$ and use the techniques developed there and by \cite{M} to prove an ambient Dacarogna-Moser Theorem.
\end{abstract}

\section{Introduction and Notation}
This paper will be divided into three parts. The first part will give a brief introduction to symplectic geometry. If one wants a more comprehensive introduction, we refer them to \cite{SM} or \cite{DS}. The subject of this and the next section is \textit{Moser's Method}, a method commonly seen in Darboux's theorem in symplectic geometry. Moser's trick can be summed up by the following principle: \\\\
\textit{When trying to find a diffeomorphism, try constructing an isotopy for the diffeomorphism, and to construct such an isotopy, find the vector field associated with the isotopy}. \\\\
There we will present the \textit{Darboux Theorem} and Moser's homotopy method which was first found in \cite{M}. For readers who wish to have a more comprehensive introduction to symplectic geometry, we recommend \cite{SM} and \cite{DS}. \\\\
The second part will contain a proof of a special case of \thref{StrongDarboux}, the \textit{Strong Darboux Theorem} which can be found in \cite[Theorem 21.1.6]{Hor}. The main result is a proof of the following using Moser's Homotopy Method.
\begin{theorem}\thlabel{OneVariableDarboux}
Let $p: \mathbb{R}^{2n} \to \mathbb{R}$ be a smooth function with $\frac{\partial p}{\partial x_1}(x, y) > 0$ for all $(x, y)$ in some open set $O$. Then there exists a neighborhood $U$ in $\mathbb{R}^{2n}$ of $(x, y)$ and a symplectomorphism $f: U \to \mathbb{R}^{2n}$ such that $f^*p = x_1$.
\end{theorem}
The third and final part of the paper concerns the extension of diffeomorphisms defined on closed submanifolds or subsets. The first two results we will prove are 
\begin{theorem}\thlabel{permute}
Let $S^1$ be considered as a submanifold of $\mathbb{R}^2$. Let $f: S^1 \to S^1$ be a $C^{k + 1}$ (with $k > 0$ an integer) diffeomorphism. Then $f$ extends to a $C^k$ symplectomorphism $\varphi: \mathbb{R}^2 \to \mathbb{R}^2$. Furthermore, the extension $\varphi$ near a neighborhood of $S^1$ is of the form $\varphi(r, \theta) = (g(r, \theta), f(\theta))$ for some $g$.
\end{theorem}
\begin{theorem}\thlabel{Circle}
Let $f:S^1 \to S^1$ be a $C^k$ diffeomorphism with $k \ge 1$ an integer. Then there exists a $C^k$ symplectomorphism $\varphi: \mathbb{R}^2 \to \mathbb{R}^2$ such that $\varphi|_{S^1} = f$.
\end{theorem}
\thref{permute} will be proven using Moser's homotopy method while \thref{Circle} will be proven using the method of generating functions. It is remarked that it is not possible to prove a version of \thref{Circle} where we specify the extension to be of the form $(g(r, \theta), f(\theta))$ (see e.g. \cite[Chapter 2]{Z} or \thref{Uniform}). Then we will prove an extension result similar to \thref{permute} but for the square (see notation section below for the definition of a square). Let $S \subset \mathbb{R}^2$ denote the square, the boundary of the set $[0, 1]^2$. Any diffeomorphism of the square (see notation section below for the definition of a diffeomorphism of the square) can be identified as four diffeomorphisms from the unit interval to itself, each representing a side of the square. Since a diffeomorphism of the square extends to a diffeomorphism of all of $\mathbb{R}^2$, it follows that its derivative extends to the vertices of $S$. We thus have the following:
\begin{theorem}\thlabel{Singularity}
Let $k \ge 0$ be an integer. Suppose $\varphi: S \to S$ is a $C^{k + 1}$ diffeomorphism whose derivative is equal to $1$ on each of the vertices of $S$. Then there exists an ambient $C^{k}$ diffeomorphism $\psi: \mathbb{R}^2 \to \mathbb{R}^2$ such that $\varphi = \psi$ on $S$ and $\psi$ is area preserving.
\end{theorem}
 Using the arguments for the proof of \thref{Singularity} and \cite{M}, we will prove an ambient Dacarogna-Moser's Theorem. The Dacarogna-Moser Theorem (first found in \cite{DM}) is the following:
\begin{theorem}[Dacarogna-Moser Theorem]
Let $\Omega$ be an open set in $\mathbb{R}^n$ with a $C^{k + 3}$ boundary $f, g \in C^{k, \alpha}(\overline{\Omega}; \mathbb{R}^+)$ with $\alpha > 0$. If we define $\tau = f dx_1 dx_2 \dots dx_n$, $\omega = g dx_1 dx_2 \dots dx_n$, $\lambda = \int \tau /\int \omega$, then there exists a $v \in \rm{Diff}^{k + 1, \alpha}(\overline{\Omega})$ (see notation at the end of the section) such that $\lambda \tau = v^*\omega$ and $v|_{\partial \Omega} = Id$. 
\end{theorem}
Here, we will prove:
\begin{theorem}\thlabel{SmoothBoundary}
Let $U$ be a bounded connected open set in $\mathbb{R}^n$ with connected $C^\infty$ boundary. Suppose $f,g \in C^k(\mathbb{R}^n; \mathbb{R}^+)$. Then there exists a $C^k$ diffeomorphism $\varphi: \mathbb{R}^n \to \mathbb{R}^n$ such that $\det(d\varphi)f \circ \varphi = (\int_U f/\int_U g) g$ in a neighborhood of $\overline{U}$ and $\varphi|_{\partial U} = Id$.
\end{theorem}
\begin{remark}
A shorter albeit non-elementary proof of the above theorem for the case of $U$ simply connected can be found in \cite[Theorem 9.6]{CDK}.
\end{remark}
The arguments used to prove \thref{Singularity} and \thref{SmoothBoundary} are based on the arguments of \cite{M} and a separation of variables argument found in \cite[Proposition 10]{DM}. Note that \cite{BK} and \cite{Mc} showed that it is not possible to gain a degree of regularity for \thref{SmoothBoundary} if one does not assume additional H\"older regularity on $f$ and $g$. If one could modify the construction in \thref{mose2} so that it holds for a domain such as $B = [-1, 1]^2$ with boundary conditions on $\partial B \cup [-1, 1] \times \{0\} \cup \{0\} \times [-1, 1]$, then one could generalize \thref{Singularity} for all polygons. \\\\
We will use the following as notation throughout the paper:
\begin{itemize}
    \item If we have specified a metric space to work in, $B_r(x)$ is the open ball of radius $r$ around $x$.
    \item The coordinates of $\mathbb{R}^{2n}$ are $\mathbb{R}^{2n} = (x_1, \dots, x_n, y_1, \dots, y_n)$, and if left unspecified, $\mathbb{R}^{2n}$ has the symplectic form $\omega_0 = \sum_{i = 1}^n dx_i \wedge dy_i$.
    \item If $(M, \omega)$ is specified to be a symplectic manifold, then $M$ is a smooth manifold and $\omega$ is a non-degenerate closed two form on $M$.
    \item If $X_t$ is a vector field, then the flow $\varphi_t$ of the vector field of $X$ is the function satisfying
    $$\frac{d}{dt}\varphi_t = X_t \circ \varphi_t$$
    $$\varphi_0 = \text{Id}.$$
    \item The symplectic form on the cylinder $\mathbb{R} \times S^1 = (s, \theta)$ or $\mathbb{R}^+ \times S^1$ is $ds \wedge d\theta$.
    \item For a function $F: \mathbb{R}^n \to \mathbb{R}$, we denote $F(x_1, x_2, \dots, x_{i - 1}, r, x_{i + 1}, \dots, x_n)$ as $F(x_i = r)$.
    \item An orientation preserving diffeomorphism $f: S^1 \to S^1$ will be identified with its lift $F: \mathbb{R} \to \mathbb{R}$ such that $F(x + 1) = F(x) + 1$.
    \item The coordinates of a cylinder $S^1 \times \mathbb{R}$ will be denoted $(\theta, s)$.
    \item If $f: \mathbb{R}^n \times \mathbb{R}^n \to \mathbb{R}$ and $(x, y) \in \mathbb{R}^n \times \mathbb{R}^n$, then $f_{x_i}$, $f_{y_i}$ will be used to denote $\frac{\partial f}{\partial x_i}$ and $\frac{\partial f}{\partial y_i}$, respectively.
    \item $\text{Diff}^k(S)$ will denote $C^k$ diffeomorphisms of $S$. If $S$ is closed, then this will mean that all $k$ derivatives exist and have continuous extension to the boundary.
    \item $C^k_c(U)$ will denote $C^k$ compactly supported functions on $U$.
    \item $\mathbb{R}^+ = (0, \infty)$.
    \item The square is $\partial [0, 1]^2$ and will often be denoted with an $S$.
    \item The set $\text{Diff}^k(S)$ consists the set of all $\varphi$ such that there exists $\psi \in \text{Diff}^k(\mathbb{R}^2)$ with $\psi|_{S} = \varphi$.
\end{itemize}

\subsection{Acknowledgements:} This paper is my undergraduate thesis for UC Berkeley. I am grateful to Charles Pugh and Maciej Zworski for advising this project and for many helpful comments. I am also grateful to Cesar Silva for Charles Pugh and Cesar Silva for introducing me to the field of dynamical systems. I would like to thank Wilfrid Gangbo for his prompt response and helpful comments on this paper. In addition, I would like to thank Aidan Backus, Mason Haberle, and Reuben Drogin for helpful discussions regarding both this paper and many other mathematical topics throughout my time at UC Berkeley. Finally, I would like to thank Conan Wu for maintaining her blog. In particular, \cite{Wu} was one of the primary motivations for this project.

\section{Background}\label{Introduction}
\subsection{Definitions and Examples}
Let $M$ be a manifold. A \textit{symplectic form} on $M$ is a non-degenerate two form $\omega$ such that $d\omega = 0$. 
\begin{example}
The manifold $\mathbb{R}^2$ with the form $\omega = dx \wedge dy$ is a symplectic manifold. More generally, $\mathbb{R}^{2n}$ with the form $dx \wedge dy = \sum_{i = 1}^n dx_i \wedge dy_i$
\end{example}
\begin{example}
Similarly, the Torus $\mathbb{T}^{2n}$ is a symplectic manifold with symplectic form $dx \wedge dy$ and the cylinder $S^1 \times \mathbb{R}$ is a symplectic manifold with the form $dr \wedge d\theta$ where $r$ is the $\mathbb{R}$ variable and $\theta$ is the $S^1$ variable. 
\end{example}
A diffeomorphism $f: (M_1, \omega_1) \to (M_2, \omega_2)$ between two symplectic manifolds $(M_1, \omega_1)$ and $(M_2, \omega_2)$ is known as a \textit{symplectomorphism} if $f^*\omega_2 = \omega_1$.
\begin{example}
Consider the sphere $S^2$ with the symplectic form $\sin(\varphi) d\theta \wedge d\varphi = \omega_1$ where the sphere has spherical coordinates $(\varphi, \theta)$ whose cartesian coordinates are 
$$(\sin(\varphi)\sin(\theta), \cos(\varphi)\sin(\theta), \cos(\theta)),$$
and the cylinder $S^1 \times (-1, 1)$ with the symplectic form $dr \wedge d\theta = \omega_2$. Then the Archimedes projection defined by $(r, \theta) \mapsto (\cos^{-1}(r), \theta)$ is a symplectomorphism from the cylinder $S^1 \times (-1, 1)$ to $S^2 \setminus \{(0, 0, 1), (0, 0, -1)\}$.
\end{example}
Let $(M, \omega)$ be a symplectic manifold and $f \in C^\infty(M)$ a smooth function. The \textit{Hamiltonian flow} for $f$, denoted $H_f$ is the unique vector field such that $df(X) = \omega(X, H_f)$. In local coordinates (i.e. when our manifold is $\mathbb{R}^{2n}$ and our symplectic form is $\sum_{i = 1}^n dx_i \wedge dy_i$), we can compute that
$$H_f = \sum_{i = 1}^{n} \frac{\partial f}{\partial y_i}\frac{\partial }{\partial x_i} - \frac{\partial f}{\partial x_i}\frac{\partial }{\partial y_i}.$$
\begin{theorem}
Let $\varphi_t$ be defined as the flow of the vector field $H_f$. Then $\varphi_t$ is a local symplectomorphism of $M$.
\end{theorem}
\begin{proof}
We wish to prove that $\varphi_t^*\omega = \omega$. This follows from the Cartan magic formula:
$$\varphi_t^*\omega = \omega \iff \frac{d}{dt} \varphi_t^*\omega = 0$$
$$\frac{d}{dt}\varphi_t\omega = \mathcal{L}_{H_f}\omega = i_{H_f}d\omega + d(i_{H_f}\omega) = d(df) = 0.$$
\end{proof}
The final thing we will consider in this section before moving on to Moser's Method is the notion of a \textit{generating function}. A $C^2$ function $S: M \to \mathbb{R}$ is known as a generating function for the (local) symplectomorphism $(x, S_x) \mapsto (y, -S_y)$. In order for this to be a local symplectomorphism, we must have the requirement that $\frac{\partial S}{\partial x \partial y}$ is nondegenerate. We compute that the map defined is indeed a symplectomorphism:
$$dy \wedge d(-S_y) = S_{xy} dx \wedge dy$$
$$dx \wedge d(S_x) = S_{xy} dx \wedge dy.$$
There is actually another convention one can use for generating functions: namely, $(x, S_x) \mapsto (S_y, y)$. Indeed, since we swapped the two terms in the range of the map, we have to multiply one of the factors by $-1$ to cancel out the effect of the swap reflecting the fact that $\omega \wedge \nu = -\nu \wedge \omega$ for any forms $\omega$ and $\nu$. 

\subsection{Darboux's Theorem}

One fundamental theorem in symplectic geometry is known as Darboux's Theorem. The standard formulation of Darboux's theorem is as follows:
\begin{theorem}\thlabel{Darboux}
Let $M$ be a $2n$ dimensional smooth manifold, $N$ a compact submanifold. Suppose $\omega_1, \omega_2$ are closed two-forms that are equal and nondegenerate on $N$. Then there exists neighborhoods $U_0$ and $U_1$ of $N$ and a symplectomorphism $f: (U_0, \omega_1) \to (U_1, \omega_2)$ such that $f|_{N} = \text{id}$.
\end{theorem}
The proof the theorem follows \textit{Moser's Method} which was first found in \cite{M}. The below presentation follows that in \cite{SM}.
\begin{proposition}\thlabel{MoserHomotopy}
Let $M$ be an $n$-dimensional smooth manifold, $t \in [0, 1]$ such that $\omega_t$ a family of non-degenerate two forms with an exact derivative:
$$\frac{d}{dt}\omega_t = d\sigma_t.$$
Then there exists for small $t$ a family of diffeomorphisms $\psi_t: M \to M$ such that $\psi_t^* \omega_t = \omega_0$.
\end{proposition}
\begin{proof}
Instead of trying to find the diffeomorphisms $\psi_t$, we find the vector field that this diffeomorphism generates:
$$\frac{d}{dt} \psi_t = v_t \circ \psi_t, \psi_0 = \text{Id}.$$
We can then use Picard's theorem to solve for $\psi_t$ for small time and this automatically ensures that $\psi_t$ are diffeomorphisms of $M$ since $\psi_t$ is just the time $t$ flow along the integral curves of the vector field. The integral curves foliate the manifold ensuring that $\psi_t$ are indeed bijective. Since
$\psi_0 = \text{Id}$, the condition
$$\frac{d}{dt} \psi_t^*\omega_t = \omega_0 \iff \frac{d}{dt} \psi_t^* \omega_t = 0.$$
By Cartan's magic formula and the chain rule, it follows that we have the following relations:
\begin{align*}
\rho_t^* \omega_t = \omega_0 &\iff \frac{d}{dt}\rho_t^* \omega_t = 0 \\
&\iff \rho_t^*\left(\mathcal{L}_{v_t}\omega_t + \frac{d}{dt}\omega_t\right) = 0 \\
&\iff \mathcal{L}_{v_t}\omega_t + d\sigma_t = 0 \\
&\iff d i_{v_t} \omega_t + d\sigma_t = 0 \\
&\iff d(i_{v_t} \omega_t + \sigma_t) = 0
\end{align*}
Thus, it suffices to find $v_t$ such that $i_{v_t} \omega_t = -\sigma_t$. Since $\omega_t$ is nondegenerate, the family of vector fields is unique. 
\end{proof}
\begin{proof}[Proof of \thref{Darboux}]
To prove \thref{Darboux}, we apply the argument of \thref{MoserHomotopy} to $\omega_t = t\omega_1 + (1 - t)\omega_0$. However, there are several points we must establish before we apply this argument:
\begin{itemize}
    \item[(1)] The derivative $\frac{d}{dt} \omega_t = \omega_1 - \omega_0$ is exact in a neighborhood of $N$.
    \item[(2)] The vector field $v_t$ gives a unique flow for all time $t \in [0, 1]$ for a neighborhood of $N$ given that $v_t = 0$ on $N$ and $v_t$ has small derivative near $N$.
    \item[(3)] $\omega_t$ is nondegenerate for all $t \in [0, 1]$ on a neighborhood of $N$.
\end{itemize}
All 3 statements are true when we consider the entire thing locally: namely that they should hold if we are working in open neighborhoods in $\mathbb{R}^{2n}$ and where $N$ is the $x$-axis. (3) becomes obvious under this context. (2) follows careful analysis of the proof of Picard's theorem, namely the condition that if our ODE $x' = f(x), x(t_0) = x_0$ is defined on $[t_0 - a, t_0 + a] \times B_b(\ell_0)$, and if $a < C\min(b/(||f||_{C^0(B_a(t_0) \times B_b(x_0))}), 1/\text{Lip}(f))$ for some constant $C$, then $x' = f(x)$ has a unique solution in time $[t_0 - a, t_0 + a]$. In our case, $t_0 = 0$, $\ell_0$ a portion of the $x$-axis. Letting $a = 1$, we can still shrink $b$ so that $v_t$ is sufficiently small and has small derivative so that the solution exists for all time in $[0, 1]$. It thus remains to show (1). Consider the exponential map
$$\text{exp}: T^\perp N \to M$$
which maps the normal bundle of $N$ to a tubular neighborhood of $N$ in $M$. Let $\varphi_t = \text{exp}(\cdot, t\cdot)$. Define $X_t$ to be the vector field generated by $\varphi_t$ and $H: \Omega^2(M) \to \Omega^1(M)$ via
$$H(\omega) = \int_0^1 \varphi_t^*(i_{X_t}\omega dt).$$
Now since
$$\frac{d}{dt} \varphi_t^*\omega = \rho_t^*\mathcal{L}_{X_t}\omega = \varphi_t^*(i_{X_t}d\omega + d(i_t\omega_t)) = d\varphi_t^*i_t\omega$$
it follows that
$$dH(\omega) = \int_0^1 \frac{d}{dt} \varphi_t^*\omega dt = \varphi_1^*\omega - \varphi_0^*\omega = \omega.$$
Hence, letting $\omega = \omega_1 - \omega_0$, and $H(\omega) = \sigma_t$, we can apply Moser's argument restricted to a neighborhood of $N$ to conclude our argument.
\end{proof}

\section{The Strong Darboux Theorem}
The usual Strong Darboux Theorem, which can be found, for instance, in \cite[Theorem 21.1.6]{Hor}, states the following:
\begin{theorem}\thlabel{StrongDarboux}
Let $M$ be a $2n$ dimensional symplectic manifold, $A$, $B$ two subsets of $\{1, \dots, n\}$. Let $x_0$ be a point in $M$, $q_i$, $i \in A$, $p_j$, $j \in B$ be $C^\infty$ functions on $M$ with linearly independent differentials satisfying $\{p_i, q_j\} = \delta_{ij}$, $\{p_i, p_j\} = 0$, $\{q_i, q_j\} = 0$. Then there exists local coordinates $(x_i, y_i)$ near $x$ such that $x_i = q_i$, $y_j = p_j$.
\end{theorem}
There is a proof of this using Frobenius's theorem. We shall prove a special case of this using Moser's homotopy method:

\begin{proof}[Proof of \thref{OneVariableDarboux}]
For a differentiable function $\ell$, we shall use $\ell_{x_i}$ to denote differentiation in $x_i$, $\ell_{y_i}$ to denote differentiation in $y_i$. Since $p_{x_1} \neq 0$, we can find $g: U \to \mathbb{R}^{2n}$ such that $g(x, y) = (h(x, y), x', y)$ where $x' = (x_2, \dots, x_n)$ and $p(g(x, y)) = x_1$. This is because if we let $q(x, y, z) = p(x, y) - z$, where $z \in \mathbb{R}$, then since $q_{x_1} = p_{x_1} \neq 0$ near some $x_0$, we can write $x_1 = h(z, x', y)$ for $(z, x', y) \in U$ for some $U$ and for smooth $h$ by the implicit function theorem. Treating $z$ as the first coordinate of $\mathbb{R}^{2n}$, we obtain $p(h(x, y), x', y) = x_1$ locally. \\\\
Let 
$$\omega_1 = g^* \omega_0 = \sum_{i = 1}^n h_{x_i} dx_i \wedge dy_1 + \sum_{i = 2}^n h_{y_i} dy_i \wedge dy_1 + \sum_{i = 2}^n dx_i \wedge dy_i.$$ 
Letting 
$$\mu_0 = \sum x_i dy_i,$$
we have $\omega_1 - \omega_0 = g^* \omega_0 - \omega_0 = (g^* - \text{Id}^*)\omega_0 = (g^* - \text{Id}^*)d\mu_0 = d((g^* - \text{Id}^*)\mu_0) = d((h - x_1) dy_1) = d\mu$ where we define $\mu = (h - x_1) dy_1$. In view of Moser's method, we want to find $v_t$ such that $v_t(x_1) = 0$ and
$$d(i_{v_t}\omega_t + \mu) = 0.$$
When unambiguous, we shall write an such as $v_t$ that depends on $t$ as just $v$, omitting the time $t$. Letting 
$v = \sum_{i = 1}^n \left(v^{x_i} \frac{\partial}{\partial x_i} + v^{y_i} \frac{\partial}{\partial y_i}\right) = (v^x, v^y) = (0, v^{x'}, v^y)$, we have
$$i_{v} \omega_t = -v^{y_1}\left(1 - t + th_{x_1}\right)dx_1 - \sum_{i = 2}^n \left(v^{y_1}th_{x_1} - v^{y_i} \right) dx_i$$
$$+ t\sum_{i = 2}^n \left(v^{y_i} h_{y_i} + v^{x_i}h_{x_i}\right)dy_1 - \sum_{i = 2}^n \left(v^{y_1} th_{y_i} - v^{x_i} \right) dy_i$$
Thus, it suffices to solve for $v_t$ and $k_t$ in the following equation. 
\begin{equation}
\mu + i_{v}\omega_t = dk_t = \sum_{i = 1}^n k_{x_i} dx_i + k_{y_i} dy_i. \tag{1}    
\end{equation}
This is equivalent to
\begin{equation}
v^{x_1} = 0\tag{2} 
\end{equation}
\begin{equation}
v^{y_1} = -\frac{k_{x_1}}{1 - t + th_{x_1}} \tag{3}    
\end{equation}
\begin{equation}
v^{x_i} = -\frac{tk_{x_1}h_{y_i}}{1 - t + th_{x_1}} + k_{y_i}, \text{ for } 2 \le i \le n \tag{4}
\end{equation}
\begin{equation}
v_t^{y_i} = -\frac{tk_{x_1}h_{x_i}}{1 - t + t h_{x_1}} + k_{x_i}, \text{ for } 2 \le i \le n. \tag{5}    
\end{equation}
\begin{equation}
t\sum_{i = 2}^n \left(v^{y_i} h_{y_i} + v_t^{x_i}h_{x_i}\right) + h - x_1 = k_{y_1}. \tag{6}    
\end{equation}
where (2) is by assumption on $v$, (3) is obtained from equating the two $dx_1$ terms in (1), (4) is obtained from equating the $dy_i$ terms in (1) and by (3), (5) is obtained from equating the $dx_i$ terms in (1) and (3), and (6) is obtained from equating the $dy_1$ terms in (1). These equations are well defined since $h_{x_1} > 0$ because $g_{x_1} > 0$ and since $\frac{d}{dx}f^{-1}(x) = \frac{1}{f^{-1}(f'(x))}$. From the combination of (5), (4), and (3), we obtain from (5)
\begin{equation}
\nabla k_t \cdot n = h - x_1 \tag{6}    
\end{equation}
where
$$n = \left(\sum_{i = 2}^n \frac{t^2 h_{y_i}^2 + h_{x_i}^2}{1 - t + th_{x_1}}, -th_{y_2}, -th_{y_3}, \dots, -th_{y_n}, 1, -th_{x_2}, -th_{x_3}, \dots, -th_{x_n}\right).$$
Thus, a solution of $k$ in (6) yields a solution for $v$ and by Moser's method a symplectomorphism that preserves the $x_1$ coordinate.
We now use the method of characteristics to solve this formula and find that the PDE admits a local solution: we find $p = (p_x, p_y)$ such that
$$\frac{dp^{x_i}}{ds} = n^{x_i}(p)$$
$$p(s = 0) = (x_0, y_0)$$
where $n^{x_i}$ is the $x_i$th coordinate of $n$. Then we can obtain a local solution to the system of ODEs and by shrinking the neighborhood, we obtain a solution for all time in $[0, 1]$. Since $\frac{\partial h}{\partial x_1} > 0$. Thus, there exists a vector field $v$ such that $v^{x_1} = 0$ and it follows that there exists $\rho_t$ such that $\frac{d}{dt} \rho_t = v_t \circ \rho_t$ and $x_1(\rho_t) = x_1$, and $\rho_t^* \omega_t = \omega_0$. Setting $f = \rho_1 \circ g$ and $U$ be the neighborhood we constructed when we inverted in $x_1$, we have our desired symplectomorphism.
\end{proof}

\begin{corollary}
Suppose $p: \mathbb{R}^{2n} \to \mathbb{R}$ is a smooth function with $\nabla p \neq 0$. Then there exists an open neighborhood $U$ of $\mathbb{R}^{2n}$ and a symplectomorphism $f: U \to \mathbb{R}^{2n}$ such that $f^*p = x_1$ on $U$.
\end{corollary}
\begin{proof}
There is a neighborhood for which $\nabla p$ is nonzero. Translate $p$ so that $\nabla p$ is nonzero in a neighborhood of the origin. There exists a neighborhood $U$ and a unitary matrix $M$ such that $p(M)$ has positive derivative at $x_1$ in some open set $O$ inside $U$. To see this, rotate the input until we get a nonzero derivative in $x_1$ near some neighborhood. If the derivative in $x_1$ is negative everywhere, reflect across the origin (e.g. $x \mapsto -x$) to ensure that the derivative in $x_1$ is positive. Then apply the theorem above with $p = p \circ M$ with $O$ the neighborhood.
\end{proof}

\section{Area Preserving Diffeomorphisms}

\subsection{Extending a Diffeomorphism of $S^1$}
Here, we will prove a conservative extension result for $S^1$. The results below are inspired by \cite{Wu} and according to \cite{Wu}, to have been known by the authors of \cite{BCW} and \cite{BGV}. 
Note: in this section, a function $f$ on $S^1$ will be identified with its lift $F: \mathbb{R} \to \mathbb{R}$ satisfying $F(x + 1) = F(x) + 1$. Its derivative is then a periodic function on $\mathbb{R}$ with period $1$ (instead of the usual convention of $2\pi$)

\begin{proof}[Proof of \thref{permute}]
Suppose without a loss of generality that $f$ has positive orientation, or that $f'(\theta) > 0$ for each $\theta$. Let $\omega_0$ denote the standard symplectic form $ds \wedge d\theta$ on the cylinder. We first extend $f$ to a diffeomorphism of $\mathbb{R}^2$ to itself such that it is the identity outside a neighborhood of the circle. The map $f \times I$ is a diffeomorphism except at the origin, where it is not smooth. Let $\chi: \mathbb{R} \to \mathbb{R}$ be a nonnegative $C^\infty$ monotone smooth function equal to $0$ on a neighborhood of $1$ and equal to $1$ outside some neighborhood not containing the origin. Then $g(s, \theta) = (s, (1 - \chi(s))f(\theta) + \chi(s)\theta) = (s, h(s, \theta))$ is a diffeomorphism of $\mathbb{R}^2$ to itself that extends $f$ on $S^1$. It is evident that it is a smooth $C^{k + 1}$ function. The only thing to check is that it is a diffeomorphism. Since $g$ preserves each of the $s$-lines, it suffices to check that it is bijective and has nonzero derivative in its $\theta$-term. Injectivity follows from $\frac{\partial h}{\partial \theta}$ is always positive and since $\int_0^1 (1 - \chi(s))f(\theta) + \chi(s)\theta d\theta = (1 - \chi(s)) + \chi(s) = 1$ (so the map is degree $1$). Surjectivity follows from the latter statement of the degree of the map $h$ in $\theta$ is nonzero. Hence, the function is injective. To check surjectivity, it suffices to check on level sets of $s$. 
Since $g$ is $C^{k + 1}$, we have $g^* ds \wedge d\theta = \frac{\partial h}{\partial \theta} ds \wedge d\theta = \omega_1$ is a $C^{k}$ differential form. Let $\omega_t = t\omega_1 + (1 - t) \omega_0$ We apply Moser's argument here to $\omega_0$ and $\omega_1$ via finding a vector field $v_t$ defined on $[0, 1]$ such that
\begin{equation}
\frac{d}{dt} \rho_t = v_t (\rho_t), \rho_t^*\omega_t = \omega_0, \rho_0 = \text{Id} \text{, } v_t^\theta = 0 \text{, and } v_t(s = 1) = 0 \tag{1}.
\end{equation}
Letting $\mu = (h - \theta) ds$, it suffices to find $v_t$ such that
\begin{equation}
i_{v_t} \omega_t + \mu = dk = \frac{\partial k}{\partial s} ds + \frac{\partial k}{\partial \theta} d\theta. \tag{2}   
\end{equation}
Then
\begin{equation}
i_{v_t} \omega_t + \mu = v_t^s\left(t\frac{\partial h}{\partial \theta} + 1 - t\right) d\theta + (h - \theta) ds \tag{3}    
\end{equation}
so
\begin{equation}
k(s, \theta) = \int_1^s [h(p, \theta) - \theta] dp \tag{4}   
\end{equation}
so
\begin{equation}
v_t^s = \frac{1}{t\frac{\partial h}{\partial \theta} + 1 - t}\int_1^s \frac{\partial h}{\partial \theta}(p, \theta) - 1 dp. \tag{5}    
\end{equation}
Notice that since $\frac{\partial h}{\partial \theta} = (1 - \chi(s))f'(\theta) + \chi(s) > 0$, the denominator of (5) is nonzero. As $h$ is $C^{k + 1}$, $\frac{\partial h}{\partial s}$ is $C^k$, so $v$ is $C^k$, so the flow generated by $v$ is $C^k$. Notice that outside a small neighborhood of the circle, $v_t$ is $0$, so the isotopy is the identity outside a neighborhood of the circle. Thus, we may be conjugate with the symplectomorphism from the cylinder to the plane $(s, \theta) \mapsto (s\cos(\theta), s\sin(\theta))$ minus the origin and extend it to the entire plane by sending the origin to itself. This is $C^k$ since it is smooth on the origin because it is the identity near the origin. It locally preserves the normal bundle because $\chi$ is identically equal to $0$ near a neighborhood of $S^1$ and since our vector field $v$ does not move in the $\theta$ direction.
\end{proof}

The proof given here uses a construction that is essentially identical to one given in \cite{BCW} and \cite{BGV} and the proof here closely follows the proof given in \cite{BCW}.
\begin{remark}\thlabel{Uniform}
Let $M_1$ and $M_2$ be two $n$-dimensional smooth manifolds. If one is given a diffeomorphism $f: M_1 \to M_2$, one can extend this to a symplectomorphism $(f, (df^*)^{-1})$ of the cotangent bundles $T^*M_1 \to T^*M_2$. This construction is essentially the unique symplectomorphism which takes the fibers of $T^*M_1$ to $T^*M_2$. However, it loses a derivative. 
\end{remark}
\begin{proof}[Proof of \thref{Circle}]
We shall construct the symplectomorphism as a map $\varphi$ on $\mathbb{R}^2 = (x, y)$ such that $\varphi(x + 1, y) = \varphi(x, y) + 1$ and outside a neighborhood of the $x$-axis it is the identity so it is a symplectomorphism on the cylinder. Then since it is the identity outside the $s = 0$ circle, we can conjugate it with the symplectomorphism between the cylinder and the plane minus the origin via $(s, \theta) = ((s + 1)\cos(\theta), (s + 1)\sin(\theta))$ and extend it to the origin by sending the origin to itself. Then since the map is the identity near the origin, it is sufficiently smooth. Instead of constructing the symplectomorphism like we constructed above, we construct a generating function for the symplectomorphism. Recall that a generating function $S$ is a $C^{k + 1}$ function such that the matrix $\frac{\partial^2S}{\partial x \partial y}$ is locally invertible. If we have a function $S$ such that $\frac{\partial^2S}{\partial x \partial y}$ is globally invertible, it will only define a local symplectomorphism, because we have our usual bijectivity issue when we apply the implicit function theorem. However, as we shall see in \thref{global}, if $S$ has finite distance from $S_0$, then the generating function gives a global symplectomorphism. \\\\
Instead of following the usual convention that the generating function defines a local symplectomorphism
$$\varphi \left(x, \frac{\partial S}{\partial x}(x, y)\right) = \left(y, -\frac{\partial S}{\partial y}(x, y) \right)$$
we follow the convention that
$$\varphi \left(x, \frac{\partial S}{\partial x}(x, y)\right) = \left(\frac{\partial S}{\partial y}(x, y) , y\right).$$
This occurs when we invert $\nu$ in the equation $\varphi(x, y) = (\xi, \nu)$ so we end up with coordinates on $(x, \nu)$. Thus, the generating function $S_0$ of the identity function on $\mathbb{R}^2$ is $xy$.
\begin{lemma}\thlabel{global}
Let $S: \mathbb{R}^{2n} \to \mathbb{R}^{2n}$ such that $\det(d_xd_yS) \neq 0$. Let $S_0 = \langle x, y \rangle$ denote the generating function of the identity. If $d_{C^1}(S, S_0) < \infty$, then we in fact obtain a global symplectomorphism $\varphi_S$.
\end{lemma}
\begin{proof}
We prove that we can globally invert
$$\frac{\partial S}{\partial x}(x, y)$$
in $y$. It suffices to find a global inverse, since a global inverse would be identical to local inverses, so a global inverse would have the correct regularity. Fixing $x$, this becomes a function $g$ in $y$. Then the fact that $d_{C^1}(S, S_0) < \infty$ translates to the fact that there exists a constant $L$ such that for each $y \in \mathbb{R}^{n}$, $d(g(y), y) < L$. Let $B$ be a bounded subset of $\mathbb{R}^n$. Suppose $g(y) \in B$. Then since $d(g(y), y) < L$, it follows that $y \in B_L(B)$ where $B_L(S)$ is the $L$-ball of a set $S$. Thus, the preimage $g^{-1}$ of a bounded set is bounded, and since the preimage of closed sets are closed, it follows that $g$ is proper, so by the Hadamard global inverse function theorem (see below or \cite{H}, \cite{H1}, \cite{H2}), it follows that $g$ is globally invertible. Hence, this defines a global symplectomorphism of $\mathbb{R}^{2}$.
\begin{theorem}[Hadamard Global Inverse Function Theorem]\thlabel{inverse}
Let $M$ and $N$ be smooth connected manifolds of the same dimension. Suppose $f: M \to N$ is a $C^k$ function such that $f$ is proper, the Jacobian of $f$ is nonzero, and $N$ is simply connected. Then $f$ is a $C^k$ diffeomorphism.
\end{theorem}
While we may thus globally invert $g$ in $y$ to obtain $j$, $j$ may not have the appropriate regularity. We prove that it does by proving that it coincides with a local implicit function $h$. For all $(x, y)$, we have
$$\left(x, \frac{\partial S}{\partial x}(x, h(x, y))\right) = \left(x, \frac{\partial S}{\partial x}(x, g(x, y))\right) (= (x, y))$$
so applying $(Id, g)$ to both sides, we obtain
$$\left(x, g\left(\frac{\partial S}{\partial x}(x, h(x, y))\right)\right) = \left(x, g\left(\frac{\partial S}{\partial x}(x, g(x, y))\right)\right)$$
$$(x, h(x, y)) = (x, g(x, y))$$
so the global inverse agrees with the local inverse and thus has the appropriate regularity.
\end{proof}
We now continue the proof of \thref{Circle}. Rotate the circle so that $f(0) = 0$. Fix $\epsilon > 0$ and let $\chi: \mathbb{R} \to \mathbb{R}$ be a $C^\infty$ function such that $\chi$ is supported in $(-1, 1)$, $\chi(0) = 1$, $\chi^{(k)}(0) = 0$ for all $k \ge 1$, and $\int \chi(x) dx = 1$. Let $\rho: \mathbb{R} \to \mathbb{R}$ be a $C^\infty$ function with $\text{supp}(\rho) \subseteq (-\epsilon, \epsilon)$ and $\rho \equiv 1$ on $(-\epsilon/2, \epsilon/2)$. Let $a:S^1 \to \mathbb{R}$ be defined by $a(x) = f'(x) - 1$. Because $f'$ is periodic, $a$ is well-defined on $S^1$. Note that
\begin{equation}
\int_{u}^{u + 1} a(t) dt = f(u + 1) - f(u) - (u + 1) + u = f(u) + 1 - f(u) - u - 1 + u = 0\tag{1}    
\end{equation}
for all $u$. Consider the map $Q: S^1 \times \mathbb{R} \to \mathbb{R}$ with
$$Q(x, y) = y\int_0^x \int_{-\infty}^\infty \chi(t)a(s - ty) dt ds = \text{sign}(y) \int_0^x \int_{-\infty}^\infty \chi\left(\frac{s - t}{y}\right) a(t) dt ds.$$
Let $S = xy + \rho(y) Q(x, y)$. We can calculate that
\begin{align*}
\frac{\partial Q}{\partial y} &= \int_0^x \int_{-\infty}^\infty \chi(t)a(s - ty) dtds - y\int_0^x \int_{-\infty}^\infty \chi(t)t a'(s - ty) dt ds\\
&= \int_0^x \int_{-\infty}^\infty \chi(t) a(s - ty) dt ds - \int_0^x \int_{-\infty}^\infty \chi(t)a(s - ty) dt ds - \int_0^x \int_{-\infty}^\infty \chi'(t) t a(s - ty) dt ds \\
&= -\int_0^x \int_{-\infty}^\infty \chi'(t) t a(s - ty) dt ds \\
&= -\int_{-\infty}^\infty \chi'(t) t \int_{- ty}^{x - ty} a(s) ds dt\tag{1}
\end{align*}
(where the second equality is obtained by integration by parts) and setting $y = 0$, we see that it is equal to
$$\left(-\int_{-\infty}^\infty \chi'(t)t dt \right)\left(\int_0^x a(x) dx \right) = \int_0^x a(x) dx = f(x) - f(0) - x = f(x) - x.$$
Hence
$$\frac{\partial S}{\partial y}(x, 0) = f(x).$$
Also,
\begin{equation}
\frac{\partial S}{\partial x}(x, 0) = y + y\int_{-\infty}^\infty \chi(t) a(x - ty) dt\tag{2}
\end{equation}
and setting $y = 0$ we obtain that the quantity is indeed $0$. In addition, both formulae (1) and (2) show that $Q$ has one more derivative than $f$. Furthermore,
\begin{align*}
Q(x + 1, y) &= y\int_0^{x + 1} \int_{-\infty}^\infty \chi(s)a(t - sy)ds dt \\
&= y\int_{-\infty}^\infty \chi(s) \int_{-sy}^{x + 1 - sy} a(t) dt ds \\
&= y\int_{-\infty}^\infty \chi(s) \left(\int_{-sy}^{x - sy} a(t) dt + \int_{x - sy}^{x + 1 - sy} a(t)dt \right) ds \\
&= y\int_{-\infty}^\infty \chi(s) \int_{-sy}^{x - sy} a(t) dt ds \\
&= Q(x, y)
\end{align*}
so we can calculate that
$$\frac{\partial S}{\partial x}(x + 1, y) = \frac{\partial S}{\partial x}(x, y)$$
and
$$\frac{\partial S}{\partial y}(x + 1, y) = \frac{\partial S}{\partial y}(x, y) + 1.$$
This shows that the diffeomorphism $\varphi_S$ generated by the generating function $S$ sends $(u, 0)$ to $(f(u), 0)$. Since $\frac{\partial S}{\partial x}$ is periodic in the $x$ coordinate with period $1$ and $\frac{\partial S}{\partial y}$ is has the property that evaluated at $x + 1$ is $1$ more than $x$, the diffeomorphism generated by the generating function defines a diffeomorphism on the neighborhood of $S^1$ on the plane--it defines a symplectomorphism on the cylinder, but since $S = S_0$ outside a neighborhood of $S^1$, it does define a symplectomorphism of the plane. Since $S = S_0$ outside a neighborhood of the circle, $S - S_0$ has finite $C^1$ norm, so by \thref{global}, it follows that the generating function $S$ defines a global symplectomorphism of the cylinder which is the identity outside a neighborhood of the cylinder. Hence, by conjugating by a symplectomorphism from the cylinder to the plane, it follows that there exists a symplectomorphism that extends $f$ and is equal to the identity outside a neighborhood of the identity.

\end{proof}
\begin{remark}
With a little bit of more work, as done in \cite{BCW}, we can show that
$$||S - S_0||_{C^2} \le K||a||_{C^0}.$$
for some constant $K$. 
\end{remark}
\begin{remark}
\cite{BCW} and \cite{BGV} give a higher dimensional analogue of the construction we have. However, as their theorems are different, it would be interesting to see if their method yields a higher dimesional statement of \thref{Circle} (e.g. $\mathbb{T}^n \subset \mathbb{C}^n$).
\end{remark}

\subsection{The Square}
In this section, we prove a conservative extension result for the square, which was defined in section 1 as the boundary of $(0, 1) \times (0, 1)$. Unlike the previous sections, the proof does not use any symplectic geometry. The square is not a smooth manifold and does not have a tubular neighborhood when it is embedded into $\mathbb{R}^2$. Therefore, the methods of the previous sections do not apply here. However, following \cite{DM} lemma 3, given a positive function $f$ on $\mathbb{R}^2$ which integrates to $1$ on the square and is equal to $1$ on the corners of the square, if one can solve for a vector field $X$ such that $\text{div}(X) = f - 1$ and $X = 0$ on the boundary of the square, then setting 
$$X_t = \frac{X}{t + (1 - t)f},$$
we could use Moser's method for the two forms $f(x) dx$ and the standard Lebesgue measure $dx$ with $X_t$ the vector field as in the presentation of Moser's method in \thref{MoserHomotopy}. See \cite{DM} or \cite[Proposition 9.7]{CDK} for further details. \\\\
Since the square is not a differentiable manifold, we must specify the definition of a diffeomorphism of the square. We shall view the square $S$ as a set embedded inside $\mathbb{R}^2$. 
\begin{definition}
A function $\varphi: S \to S$ is a $C^k$ diffeomorphism if there exists some $C^k$ diffeomorphism $\psi: \mathbb{R}^2 \to \mathbb{R}^2$ such that $\psi = \varphi$ on $S$. We shall refer $\psi$ as an \textit{ambient diffeomorphism} of $f$.
\end{definition}

In the proofs below, if not specified, $\epsilon = 0$ or $1$.
\begin{lemma}\thlabel{Separation}
Let $B = [0, 1]^2$ and $f > 0$ a $C^k$ function defined on $B$ which is equal to $1$ at the corners of the square $\{(\epsilon_1, \epsilon_2): \epsilon_i \in \{0, 1\}\}$. Then there exists some $v$, a $C^{k}$ diffeomorphism of $\mathbb{R}^2$ such that $\det(dv) = f$ on $\partial B$ and $v|_{\partial B} = Id$.
\end{lemma}
\begin{proof}
We write $v = v_1 \circ v_2$ where $v_1 = (u_1(x, y), y)$ and $v_2 = (x, u_2(x, y))$. We find $u_1$ such that $\det(dv_1)$ is equal to $f$ on the $\epsilon \times [0, 1]$ and such that $u_1(\epsilon, y) = \epsilon$ and $u_1(x, \epsilon) = x$ where $\epsilon = 0, 1$ and for all $x \in [0, 1]$. $u_1$ would therefore preserve the boundary of $B$. We find that $\det(dv_1)(x, y) = \frac{\partial u_1}{\partial x}(x, y)$ and so $\frac{\partial u_1}{\partial x}(\epsilon, y) = f(\epsilon, y)$. Then, by \thref{Potato} in the appendix, there exists $\beta$ a positive $C^k$ function such that $\beta(\epsilon, y) = f(\epsilon, y)$, $\int_0^1 \beta(t, y) dt = 1$ for all $y \in \mathbb{R}$ (it is only necessary for $y \in [0, 1]$ but we just construct for all $y$ just to simplify the argument a bit), and $\beta(x, \epsilon) = 1$. We then define 
$$u_1(x, y) = \int_0^x \beta(t, y) dt.$$
Then $u_1(\epsilon, y) = \epsilon$, $u_1(x, \epsilon) = x$, and $\frac{\partial u_1}{\partial x}(\epsilon, y) = f(\epsilon, y)$. Now set $\det(dv_1)^{-1}f = g$ and we repeat the same process as above for $g$ except with the roles of $x$ and $y$ reversed. Notice however that $g$ is equal to $1$ at the lines $(\epsilon, y)$ and that $\frac{\partial u_2}{\partial y}$ has to necessarily equal to $g$ at $(\epsilon, y)$ since $\frac{\partial u_2}{\partial y}(\epsilon, y) = 1$ because $v_2(\epsilon, y) = y$. Thus, $\det(dv) = \det(dv_1)\det(dv_2)$ on the boundary of $B$ (since $v_1$ and $v_2$ are the identity on $B$) and because $\det(dv_1)\det(dv_2) = f$ on $\partial B$, we conclude that $\det(dv) = f$ on $\partial B$. As $\beta$ is $C^k$, $v_1$ and $v_2$ are both $C^{k}$. Hence, $v$ is $C^k$ as well.
\end{proof}
Applying \thref{Separation} to $f^{-1}$ and observing that $v \equiv Id$ on $\partial B$, we have:
\begin{corollary}\thlabel{update}
Let $B = [0, 1]^2$ and $f$ be a positive $C^k$ function on $\mathbb{R}^2$ which is equal to $1$ at the corners of the square $\{(\epsilon_1, \epsilon_2): \epsilon_i \in \{0, 1\}\}$. Then there exists some $v$, a $C^{k}$ diffeomorphism of $\mathbb{R}^n$ such that the function $\det(dv)f \circ v$ is equal to $1$ on $\partial B$. 
\end{corollary}
The above lemma can also admits a higher dimensional generalization:
\begin{lemma}\thlabel{Hypercube}
Let $B = [0, 1]^n$ and $f > 0$ a $C^k$ function defined on $B$ which is equal to $1$ on all the codimension two hyperfaces of $B$, that is the set of all $(x_1, \dots, \epsilon_i, x_{i + 1}, \dots, x_{j - 1}, \epsilon_j, \dots x_n)$ for $\epsilon_i, \epsilon_j \in \{0, 1\}$. Then there exists some $v$ a $C^{k}$ diffeomorphism of $\mathbb{R}^n$ such that $\det(dv)f \circ v$ is equal to $1$ on $\partial B$ and $v|_{\partial B} = Id$.
\end{lemma}
\begin{proof}
We write $v = v_1 \circ v_2 \circ \cdots \circ v_n$ where $v_i = (x_1, \dots, x_{i - 1}, u_i(x_1, \dots, x_n), x_{i + 1}, \dots, x_n)$. We proceed by induction. Set $g = \det(dv_1)^{-1} \cdots \det(dv_{i - 1})^{-1} f$ and suppose that $g = 1$ on $(\epsilon_1, \dots, \epsilon_{i - 1}, x_{i}, \dots, x_n)$ for $\epsilon_j = 0, 1$ and $v_j$ for $1 \le j \le i - 1$ preserve the boundary of $B$. We prove that we can find $h = \det(v_i)^{-1}g$ such that $h = 1$ on $(\epsilon_1, \dots, \epsilon_i, x_{i + 1}, \dots, x_n)$, $u_i(x_i = \epsilon) = \epsilon$, $u_i(x_j = \epsilon, x_i) = x_i$ for $j \neq i$ so that $v_i$ preserves the boundary of $B$. Then $\det(v_i) = \frac{\partial u_i}{\partial x_i}$ and the necessary condition that $\frac{\partial v_i}{\partial x_i}(x_j = \epsilon) = g(x_j = \epsilon)$ for $\epsilon = 0, 1$ and $1 \le j \le i$ which is equivalent to $1 = 1$ (a true statement) and $\frac{\partial v_i}{\partial x_i}(x_i = \epsilon) = g(x_i = \epsilon)$. By \thref{Potato}, there exists a $C^k$ positive function $\beta$ defined on $\mathbb{R}^n$ such that $\beta(x_i = \epsilon) = g(x_i = \epsilon)$, $\int_0^1 \beta(x_i = t) dt = 1$, and for each $1 \le j \le n$, $\beta(x_j = \epsilon) = 1$. Setting
$$u_i(x) = \int_0^{x_i} \beta(x_i = t) dt$$
we have our desired conditions $u_i(x_i = \epsilon) = \epsilon$, $u_i(x_j = \epsilon) = x$ for $j \neq i$, and $\frac{\partial u_i}{\partial x_i}(x_i = \epsilon) = f(x_i = \epsilon)$. We have then constructed a diffeomorphism $v_i$ such that $v_i$ is the identity on the boundary of $B$ and $\det(dv_i) = g$ on the $x_j = \epsilon$ hyperfaces of the cube for $1 \le j \le i$. We can then take $(v_i^{-1})^*g$ and so by induction, our desired $v$ exists.
\end{proof}

Next, we will need Lemma 2 of \cite{M}:
\begin{lemma}\thlabel{mose}
Let $h, g \in C^k(\mathbb{R}^n; \mathbb{R}^+)$, $Q = [0, 1]^n$ such that $h - g = 0$ on $\partial Q$ and $\int_Q h dx = \int_Q g dx$. Then there exists a $C^k$ diffeomorphism $u: \mathbb{R}^n \to \mathbb{R}^n$ such that $g(u(x))\det(du) = h(x)$ and $u|_{\partial Q} = Id$.
\end{lemma}
\begin{proof}
First, we exhibit $v$ and $w$ such that $u = w^{-1} \circ v$ and $\det(dv) = h$, $\det(dw) = g$. We find $v$ and $w$ with $dv$ and $dw$ upper triangular. Letting $v = (v_1, \dots, v_n)$, we have
$$\text{det}(v) = \prod_{i = 1}^n \frac{\partial v_i}{\partial x_i} = h$$
and $v_i(x_1, \dots, x_{i - 1}, \epsilon) = \epsilon$ for $\epsilon = 0, 1$ for all $1 \le i \le n$. We construct $h_i$ such that
$$h = \prod_{i = 1}^n h_i, \hspace{0.1in} \int_0^\epsilon h_i(x_1, \dots, x_{i - 1}, t) dt = \epsilon$$
where $\epsilon = 0, 1$. Therefore, if we let
$$v_i(x) = \int_0^{x_i} h_i(x_1, \dots, x_{i - 1}, t) dt$$
we find $v_i$ that has the desired property. To find $h_i$ that satisfies this, we use induction. For $i = n$, we take
$$\frac{h(x)}{h_n(x)} = \int_0^1 h(x_1, \dots, x_{n - 1}, t) dt$$
so the $h_i$ exists and thus $v$ exists. We can repeat this process to construct $w$ fof $g$. It then remains to show that $v$ and $w$ agree on the boundary of the cube. We prove this by induction on $n$. It is certainly clear that this holds for $x_i = 0, 1$ for $i \neq n$. This is because at $x_i = 0, 1$, $h = g$. At $x_n = 0, 1$, $v_n = w_n = 0,1$, respectively. Since $\frac{h}{h_n} = \frac{g}{g_n}$ on $x_i = 0, 1$ for $i \neq n$, the induction step holds. The statement is true for the base case of $n = 1$ because the integral from $0$ to $1$ of $h$ is equal to the integral from $0$ to $1$ of $g$ as per the assumption of the theorem.
\end{proof}
We thus have the following:

\begin{proof}[Proof of \thref{Singularity}]
Let $\varphi_1$ be an ambient diffeomorphism of $\varphi$ and $f = \det(d\varphi_1)$. By \thref{update}, there exists a diffeomorphism $v: \mathbb{R}^2 \to \mathbb{R}^2$ such that $\det(dv) f \circ v|_S = 1$ where $S$ is the square and $v|_S = Id$. Let $\det(dv) f \circ v = g$. By \thref{mose} applied to $g$ and $1$, there exists a diffeomorphism $u: \mathbb{R}^2 \to \mathbb{R}^2$ such that $\det(du)g \circ u = 1$ and $u|_S = Id$. Letting $\psi = \varphi_1 \circ v \circ u$, we have a diffeomorphism that extends $\varphi$ that is area preserving. Note that since $\varphi_1$ is $C^{k + 1}$, $f$ is $C^k$ and so both $u$ and $v$ are $C^k$. This proves \thref{Singularity}.
\end{proof}

\subsection{An Ambient Dacarogna-Moser Theorem}
In \cite{M}, Moser proves that two volume forms on a manifold are similar to each other via a diffeomorphism: that is, given two $C^k$ volume forms $\sigma$ and $\tau$ that integrate to the same volume, there exists a $C^{k}$ diffeomorphism $\varphi$ such that $\varphi^*\tau = \sigma$. However, this result has no boundary condition on the diffeomorphism that \thref{SmoothBoundary} has. In a later paper \cite{DM}, Dacarogna and Moser prove a similar result with a boundary condition: given two $C^{k, \alpha}$ volume forms $\sigma$ and $\tau$ on a connected bounded open set $U$ with smooth boundary that integrate to the same volume, there exists a $C^{k + 1, \alpha}$ a diffeomorphism $\varphi: U \to U$ which has a continuous extension to the boundary of the manifold such that $\varphi|_{\partial U} = Id$, and $\varphi^*\tau = \sigma$. A modification of their argument (see \cite[Theorem 9.6]{CDK}) shows that the Dacarogna-Moser theorem extends to an ambient result on $\mathbb{R}^n$ if the open set $U$ is simply connected: namely a result stating that if $\sigma$ and $\tau$ are defined on all of $\mathbb{R}^n$ which integrates to the same volume on $U$ and perhaps satisfies some additional properties, then there exists some $\psi: \mathbb{R}^n \to \mathbb{R}^n$ such that $\psi^*\tau = \sigma$ and $\psi|_{\partial U} = Id$. We shall give an elementary proof of that fact and through the process obtain slightly different constraints: that the boundary of $U$ must be connected. \\\\
Furthermore, as shown by \cite{BK} and \cite{Mc}, the Dacarogna-Moser theorem is false for $k = \alpha = 0$. They show specifically that for all $\epsilon > 0$ there exists a continuous function $\gamma: [0, 1]^2 \to [1, 1 + \epsilon]$ such that there exists no bilipschiz function $h: [0, 1]^2 \to \mathbb{R}^2$ such that $\det(dh) = \gamma$. \\\\
\thref{SmoothBoundary} can be considered an ambient result of Dacarogna-Moser's Theorem without the increase in regularity as in the Dacarogna-Moser Theorem. It bears some resemblance to Theorem 5 of \cite{M}. The difference between the two is that \thref{SmoothBoundary} is an ambient result whereas Theorem 5 of \cite{M} only applies to volume forms defined on open sets. \\\\
The proof of \thref{SmoothBoundary} is similar the proof of the main Theorem in \cite{M}. One key difference, however, is \thref{Grid}, which we proved using a ``separation of variables" argument similar to one found in the proof of \cite[Proposition 10]{DM}. The main ingredients in the proof are \thref{Grid} and slightly modified versions of \cite[Lemmas 1, 2]{M}, presented here as \thref{Lemma1} and \thref{mose2}, respectively. Along the way, we establish conservative extension result for a diffeomorphism on the boundaries of two adjacent squares attached together. If not specified, $\epsilon$ or $\epsilon_j$ are equal to $-1, 0, 1$.
\begin{lemma}\thlabel{Grid}
Let $B = [-1, 1]^n$ and $f \in C^k(B; \mathbb{R}^+)$ function defined on $B$ which is equal to $1$ on all the codimension two hyperfaces of each of the side length one hypercubes in $B$ with vertices on integer coordinates, i.e. the set of all $(x_1, \dots, \epsilon_i, x_{i + 1}, \dots, x_{j - 1}, \epsilon_j, \dots x_n)$ for $\epsilon_i, \epsilon_j \in \{-1, 0, +1\}$. Let 
$$\partial'B = \partial B \cup \bigcup_{j = 0}^{n - 1} [-1, 1]^j \times \{0\} \times [-1, 1]^{n - 1 - j}.$$
Then there exists some $v$ a $C^{k + 1}$ diffeomorphism of $\mathbb{R}^n$ such that $\det(dv)f \circ v$ is equal to $1$ on $\partial' B$ and $v|_{\partial' B} = Id$. Furthermore, if $f - 1 \equiv 0$ on a neighborhood of $\partial B$, then we can construct $v$ such that $v \equiv Id$ on a neighborhood of $\partial B$.
\end{lemma}
\begin{proof}
This has essentially the same proof as \thref{Hypercube}. The difference is we use \thref{Canonical2} instead of \thref{Potato}.
\end{proof}
Since \thref{Grid} only depends on its input function on $\partial' B$ and the conditions of the function lie on the codimension two faces of the boundary, the above lemma is true if we just restrict to a closed subset of $\partial'B$ that is the boundary of an open set:
\begin{corollary}\thlabel{Double}
Let $Q = [-1, 1] \times [0, 1]^{n - 1}$ and $f \in C^k(\mathbb{R}^n; \mathbb{R}^+)$ function defined on $B$ which is equal to $1$ on all the codimension two hyperfaces of each of the side length one hypercubes in $B$ with vertices on integer coordinates, i.e. the set of all $(x_1, \dots, \epsilon_i, x_{i + 1}, \dots, x_{j - 1}, \epsilon_j, \dots x_n)$ for $\epsilon_i, \epsilon_j \in \{0, 1\}$ if $i \neq 1$ and if $i = 1$, $\epsilon_i \in \{-1, 0, +1\}$. Let 
$$\partial'Q = \partial Q \cup \{0\} \times [0, 1]^{n - 1}.$$
Then there exists $v\in \text{Diff}^k(\mathbb{R}^n)$ such that $\det(dv)f \circ v$ is equal to $1$ on $\partial' Q$ and $v|_{\partial' Q} = Id$. Furthermore, if $f \equiv 1$ on a neighborhood of $\partial Q$, then we can construct $v$ such that $v \equiv Id$ on a neighborhood of $\partial Q$.
\end{corollary}

\begin{lemma}\thlabel{mose2}
Let $h, g \in C^k(\mathbb{R}^n; \mathbb{R}^+)$, $Q = [-1, 1] \times [0, 1]^{n - 1}$ and
$$\partial'Q = \partial Q \cup \{0\} \times [0, 1]^{n - 1}.$$
Suppose $h - g = 0$ is on $\partial Q$, $\int_Q h dx = \int_Q gdx $, and $\int_{[0, 1]^n} h dx = \int_{[0, 1]^n} g dx$. Then there exists a $C^k$ diffeomorphism $u: \mathbb{R}^n \to \mathbb{R}^n$ such that $g(u(x))\det(du) = h(x)$. Furthermore, if $h - g \equiv 0$ on a neighborhood of $\partial Q$, then we can construct $u$ such that $u \equiv Id$ on a neighborhood of $\partial Q$.
\end{lemma}
\begin{proof}
This is essentially the same proof as \thref{mose}. The difference is the base case of $n = 1$. In that case, we simply take $v(x) = \int_0^x h(t) dt$. 

\end{proof}

In order to reduce to the case of a double hypercube as in \thref{Double} and \thref{mose2}, we prove a result similar to Lemma 1 in \cite{M}.
\begin{lemma}\thlabel{Lemma1}
Let $U$ be a bounded open set in $\mathbb{R}^n$, $B$ an open set with connected boundary and whose closure is in $U$ and $U_i$ a finite cover of $U$ such that each $U_i$ intersects the boundary of $B$. Let $\tau \in C^k(\mathbb{R}^n)$, $g \in C^k_c(U)$ such that 
$$\int_B g\tau dx = \int_U g\tau dx = 0.$$ 
Then there exists $(g_i)_{i = 0}^m$ such that $g = \sum_{i = 0}^m g_i$ such that $\int_U g_i\tau = \int_B g_i\tau = 0$ and $g_i \in C^k_c(U_i)$. Furthermore, there exists some constant $c$ such that $\|g_i\|_{C^0} \le c\|g\|_{C^0}$ and
$$c \le 1000 m^2\frac{\|\tau\|}{\min_{x \in U} |\tau(x)|}q$$
where 
$$q = \frac{\mu(U)}{\min_{U_i \cap U_j \neq \emptyset}\mu(U_i \cap U_j)} + \frac{\mu(B)}{\min_{B \cap U_i \cap U_j \neq \emptyset} \mu(B \cap U_i \cap U_j)}$$
\end{lemma}
\begin{proof}
Our proof closely follows the proof of \cite[Lemma 1]{M}. We will be using similar notation as \cite[Lemma 1]{M}.\\\\
Take functions $\gamma_i \in C_c^\infty(U_i; [0, \infty))$ such that $\sum_{i =0}^m \gamma_i = 1$ on a neighborhood of $\text{supp}(g)$. First, we order $(U_i)_{i = 0}^m$ such that $U_k \cap \partial B$ intersects $\bigcup_{j = 0}^{k - 1} U_i$. Pick an integer $\rho(k) < k$ such that $U_k \cap U_{\rho(k)} \cap \partial B$ is nonempty. We introduce the matrix $(\alpha_{jk})_{j = 0, k = 1}^{m}$
\[\alpha_{jk} = \begin{cases} 
1 & \text{ if }j = k \\
-1 & \text{   if }j = \rho(k) \\
0 & \text{   otherwise}
\end{cases}\]
Since each column contains exactly one $+1$ and $-1$, we have
\begin{equation}
\sum_{j =0}^m \alpha_{jk} = 0 \tag{1}
\end{equation}
for all $k$. Let $\eta_k \in C^\infty_c(U_k \cap U_{\rho(k)})$ be chosen later. We define
\begin{equation}
 g_j = g \gamma_j - \sum_{k = 1}^m \eta_k \alpha_{jk}. \tag{2}   
\end{equation}
Let $\lambda_k = \int_U \eta_k\tau dx$ and $\lambda_k' = \int_B \eta_k\tau dx$. Since the support of $\eta_k$ intersects both $B$ and $U \setminus B$, $\lambda_k$ and $\lambda_k'$ may be any real number. Integrating (2) against $\tau$, we end up with $2(m + 1)$ equations with $2m$ unknowns in $\lambda_k$ and $\lambda_k'$:
$$\sum_{k = 1}^m \lambda_k \alpha_{jk} = \int_U g \gamma_j\tau dx, \hspace{0.1in} \sum_{k = 1}^m \lambda'_k \alpha_{jk} = \int_B g \gamma_j\tau dx.$$
However, by (1) and 
$$\int_B g\tau dx = \sum_{j = 0}^m \int_B g \gamma_j\tau dx = 0, \hspace{0.1in} \int_U g dx = \sum_{j = 0}^m \int_U g \gamma_j\tau dx = 0$$
it follows that the $j = 0$ equation is negative the sum of all the other equations, so is redundant. Since $(\alpha_{jk})$ is an upper triangular matrix with $1$'s on the diagonal, it is rank $m$, so we can uniquely solve for all the $\lambda_k$ and $\lambda'_k$. \\\\
To show that $\|g_i\|_{C^0} \le c\|g\|_{C^0}$ for some constant $c$, we notice that 
$$|\lambda_k| \le \| \alpha^{-1} \| \|g\| \mu(U) \| \tau\|, \hspace{0.1in} |\lambda_k'| \le \|\alpha^{-1}\| \|g\| \mu(B)$$
so we can choose $\eta_k$ such that 
$$\| \eta_k \| \le 500\|g\|\|\alpha^{-1}\|\frac{\|\tau\|}{\min_{x \in U} |\tau(x)|} q.$$
We thus have the bound
$$c \le 1000 m\|\alpha^{-1}\|\frac{\|\tau\|}{\min_{x \in U} |\tau(x)|}q$$
and since $\alpha = I + N$ where $N$ is a matrix with $N^m = 0$, taking a Neumann series for $(I + N)^{-1}$, we see that $\|\alpha\|^{-1} \le 2m$. We get the desired inequality from this. 
\end{proof}
\begin{lemma}\thlabel{doublesquare}
Let $f, g > 0$ be $C^k$ smooth functions on $\mathbb{R}^{n}$, and $Q = [-1, 1] \times [0, 1]^{n - 1}$ with $f \equiv g$ on $\partial Q$. Let $\partial'Q = \partial Q \cup \{0\} \times [0, 1]^{n -1}$. Suppose
$$\frac{\int_{Q} f dx }{\int_{[0, 1]^n} g dx} = \frac{\int_Q f dx}{\int_{[0, 1]^n} g dx} = \lambda.$$
Then there exists $u \in \text{Diff}^k(\mathbb{R}^n)$ such that $u|_{\partial'Q} = Id$ and $\det(du)f(u) = \lambda g$.
\end{lemma}
\begin{proof}
Since $f/g = 1$ on the codimension two hyperfaces of $Q$, i.e. the set 
$$\{(x_1, \dots, \epsilon_i, x_{i + 1}, \dots, x_{j - 1}, \epsilon_j, \dots, x_n)$$ 
where $\epsilon_i, \epsilon_j \in \{0, 1\}$ if $i > 1$ and if $i = 1$, $\epsilon_i \in \{-1, 0, +1\}$. Since $\partial (\{0\} \times [0, 1]^{n - 1})$ is contained in $\partial Q$, we may apply \thref{Double} to $f/g$ to obtain some $v: \mathbb{R}^n \to \mathbb{R}^n$ such that $\det(dv)(f/g) \circ v \equiv 1$ on $\partial' Q$ and $v|_{\partial'Q} = Id$. Hence $\det(dv) f \circ v \equiv g$ on $\partial'Q$. Since $\det(dv)f(v)$ does not change the integral of $f$ over $Q$ or $[0, 1]^n$ (since is the identity on $\partial'Q$). Thus, we may apply \thref{mose2} to $\lambda^{-1}\det(dv)f(v)$ and $g$ to obtain $w$ such that $w^*v^*f = \lambda g$. We may then set $u = w \circ v$ and thus $u^*f = \lambda g$.
\end{proof}

\begin{proof}[Proof of \thref{SmoothBoundary}]
First, we modify $f$ and $g$ outside of a neighborhood of $\overline{U}$ so that 
\begin{equation}
\int_{\mathbb{R}^n} f(x)^t g(x)^{1 - t}dx = \infty \tag{1}
\end{equation}
for $t \in [0, 1]$ and both $f$ and $g$ are bounded. This can be done by taking $g = (1 - \chi) + g \chi$ and $f = (1 - \chi) + f \chi$ for $\chi \in C^\infty_c(\mathbb{R}^n; \mathbb{R}^+)$ equal to $1$ on $U$ and $0$ outside a neighborhood of $\overline{U}$.\\\\
Next, let $B$ be a large open ball centered at the origin such that $\overline{U} \subset B$. Applying \thref{Dumb} to $B$ and $U$, we obtain an open cover $(U_i)_{i = 0}^m$ satisfying the hypotheses of \thref{Lemma1} and that $U_i$ is diffeomorphic to $(-1, 1) \times (0, 1)^{n - 1}$ with $U_i \cap \partial B$ sent to $\{0\} \times (0, 1)^{n - 1}$ under the diffeomorphism. Furthermore, $m$ is fixed and the ratio
$$p = \frac{\mu(U)}{\max_{U_i \cap U_j \neq \emptyset} \mu(U_i \cap U_j)}$$
is bounded above as the radius of $B$ increases. Let $C$ denote the constant as found in \thref{Lemma1} applied to $U_i$ and $\tau = (f/g)^t g$ for $t \in [0, 1]$. Then since $p$ is bounded above, and since $U$ and $U_i \cap U$ is invariant as the radius of $V$ increases (from the hypothesis of \thref{Dumb}), 
$$C \le 1000m^2 \frac{\max_{t \in [0, 1]} \|\tau\|}{\min_{x \in \mathbb{R}^n, t \in [0, 1]} |\tau|} < D$$
for some large $D$ that does not depend on the radius of $V$ as $m$ is fixed and $\| \tau\|, (\min |\tau|)^{-1} < \infty$ (since $f$ and $g$ are bounded). Choose $\epsilon$ such that $0 < \epsilon < \frac{1}{8}D^{-1}(m + 1)^{-1}$. \\\\
Similar to page $2$ of \cite{M}, we reduce to the case of $f \tau dx$ and $g \tau dx$ with $g = 1$, $|f - 1| < \epsilon$ and $\tau dx$ a volume form with $\epsilon$ the same as the $\epsilon$ constructed above. To see this, we introduce a family of volume forms $\omega_t = (f/g)^tg dx$. Let $\delta = (1 + \epsilon)/\max|\log(f/g)|$, we obtain for two $t$ and $t'$ such that $|t - t'| < \delta$, $\omega_t = h_t g dx$ and $\omega_{t'} = h_{t'} g dx$, then $|h_t/h_{t'} - 1| < \epsilon$. Thus, it suffices to find $v \in \text{Diff}^k(\mathbb{R}^n)$ such that $v^* \omega_t = \lambda \omega_{t'}$ (for some constant $\lambda$). By replacing $dx$ with $\tau dx = (f/g)^{t}g dx$, we assume that $g = 1$ and $|f - 1| < \epsilon$. \\\\
Let $\lambda = \frac{1}{\mu_\tau(U)}\int_U f \tau dx$. We will define $h \in C^k(\mathbb{R}^n; \mathbb{R}^+)$ equal to $f$ on a neighborhood of $\overline{U}$ and satisfying
$$\int_V h \tau dx = \int_V \tau dx$$
and $h \equiv \lambda$ near $\partial V$ where $V \supset \overline{U}$ is a large open ball to be chosen later. Take $\chi, \psi \in C^\infty_c(V; [0, 1])$ with disjoint support such that $\chi$ equal to $1$ on $U$ and $0$ everywhere outside a neighborhood of $\partial U$, and $\psi$ to be chosen later.
We construct a function of the form:
\begin{equation}
h = \lambda((\chi \lambda^{-1}f + (1 - \chi))(1 - \psi) + p\psi) \tag{2}
\end{equation}
with $p \in \mathbb{R}^+$ also to be chosen later. Since the support of $\psi$ and $\chi$ is outside a neighborhood of the boundary, it follows that $h - \lambda \in C_c^k(V)$. If
$$\int \chi f \tau dx = \lambda m, \hspace{0.1in} \int \chi \tau dx = \ell, \hspace{0.1in} \int \psi \tau dx = k.$$
then integrating on $V$ both sides of (2) and setting it equal to $\lambda \int_V \tau dx$, we have
$$m + \int_V \tau dx - \ell - k + p k = \int_V \tau dx \implies p = (\ell + k - m)/k.$$
Since $\ell$, $m$ are fixed, making $V$ larger and the support of $\psi$ larger, we can make $k$ arbitrarily large (because of (1)) and so $p$ arbitrarily close to $1$. Note that
$$|h - \lambda| \le |f - \lambda| + |p\lambda - \lambda| \le |f - 1| + |1 - \lambda| + |p\lambda - \lambda| < 2\epsilon + |p\lambda - \lambda|$$
where the first inequality follows from $\min(\lambda, \min f, p\lambda) \le h \le \max(\lambda, \max f, p\lambda)$ and $f > 0$.
Choosing $V$ large enough, we may assume that $|p - 1| < \epsilon$, and since $1 - \epsilon < \lambda < 1 + \epsilon$, $|p\lambda - \lambda| < \epsilon + \epsilon^2 < 2\epsilon$. Hence, $|h - \lambda| < 4\epsilon$.  \\\\
Applying \thref{Lemma1} with our covering $U_i$, $g = h - \lambda$, and $\tau dx$ our volume form, we obtain $g_i \in C_c(U_i, \mathbb{R})$ satisfying 
$$g_0 + g_1 + \cdots + g_m = h - \lambda, \hspace{0.1in} \|g_i\| \le C\|g\|, \hspace{0.1in} \int_{U_i} g_i \tau dx = 0$$ 
Let
$$f_0 = \lambda \hspace{0.1in} f_i = \lambda + \sum_{j = 1}^i g_i.$$
Since $D > C$ and $\epsilon$ was chosen such that 
$$4\epsilon < \frac{1}{2}(m + 1)^{-1}D^{-1} < (m + 1)^{-1}D^{-1} \lambda,$$ 
$f_i \in C^{k}(\mathbb{R}^n; \mathbb{R}^+).$\\\\
Since $f_i - f_{i - 1} \in C^k_c(U_i; \mathbb{R}^+)$, there exists open $V_i$ such that $\bar{V_i} \subset U_i$, $\text{supp}(f_i) \subset V_i$. We also demand that $V_i$ be mapped smoothly via some $\varphi_i$ to the region $Q := [-1, 1] \times [0, 1]^{n - 1}$ such that $\varphi(\partial U \cap V_i) = \{0\} \times (0, 1)^{n - 1}$. Since $f_i - f_{i - 1}$ has compact support we may view $p_1 := \det(d\varphi_i)f_i \circ \varphi_i$ and $p_2 := \det(d\varphi_i)f_{i - 1} \circ \varphi_i$ as being in $C^k(\mathbb{R}^n)$ with $p_1 - p_2 \in C^k_c(Q)$. We can apply \thref{doublesquare} to $p_1$ and $p_2$ to obtain $v_i$ such that $\det(dv_i) p_2 \circ v_i = p_1$ and near the boundary of $(-1, 1) \times (0, 1)^{n - 1}$ and on the set $\{0\} \times [0, 1]^{n - 1}$, $v_i \equiv Id$. We then map the region $(-1, 1) \times (0, 1)^{n - 1}$ back to $V_i$ via $\varphi_i^{-1}$. Let $u_i = \varphi_i \circ v_i \circ \varphi_i^{-1}$. This is a well defined diffeomorphism on $\mathbb{R}^n$ since $v_i$ can be extended to equal the identity outside $V_i$. Taking $u = u_{1} \circ u_{2} \circ \cdots \circ u_m$, we see that $u$ is a $C^k$ diffeomorphism such that $\lambda \det(du) = h$. Since $h \equiv f$ on a neighborhood of $U$, $\lambda\det(du) = f$ on a neighborhood of $U$. This completes the proof.
\end{proof}

\section{Appendix}
\begin{proposition}\thlabel{Potato}
There exists a smooth $\beta \in C^\infty(\mathbb{R}^+ \times \mathbb{R}^+ \times \mathbb{R}; \mathbb{R}^+)$ such that for all $a, b \in \mathbb{R}^+$, $\beta(a, b, x)$ equals $a$ near $x = 0$ and $b$ near $x = 1$, $\beta(1, 1, x) \equiv 1$, and
$$\int_0^1 \beta(a, b, x) dx = 1.$$
\end{proposition}
\begin{proof}
Let $\chi: \mathbb{R} \to [0, 1]$ be a $C^\infty$ (symmetric) smooth function such that it is $1$ near a neighborhood of $0$ and $0$ outside $B_1(0)$. Let $C = \int_0^1 \chi(x) dx$. We construct a function of the form
\begin{equation}
\beta(a, b, x) = (a - 1 + g(a, b)) \chi\left(\frac{x}{f(a, b)}\right) + 1 - g(a, b) + (b - 1 + g(a, b)) \chi\left(\frac{x - 1}{f(a, b)}\right)\tag{1}
\end{equation}
where $f \in C^\infty(\mathbb{R}^+ \times \mathbb{R}^+; \mathbb{R}^+)$, $g \in C^\infty(\mathbb{R}^+ \times \mathbb{R}^+)$ will be constructed later. Geometrically, this function is two bump functions superpositioned at $0$ and $1$, respectively and given a value of $a$ at $0$ and $b$ at $1$. Factors $f(a, b)$ localize the bump functions at $0$ and $1$ so their supports do not intersect. A term $g(a, b) < 1$ is needed to ensure that the integral over $x$ of the above is equal to $1$. Noticing that $$\int_{\mathbb{R}^+} \chi\left(\frac{x}{f(a, b)}\right) dx = f(a, b) \int_{\mathbb{R}^+} \chi(x) dx = f(a, b)C$$
and integrating (1) over the positive reals and setting it equal to $1$, we obtain
$$Cf(a, b)(a + b - 2 + 2g(a, b)) - g(a, b) + 1 = 1$$
which rearranges as
\begin{equation}
Cf(a, b)(a + b - 2) = g(a, b)(1 - 2Cf(a, b))\tag{2}
\end{equation}
The constraint $g(a, b) < 1$ gives us an additional condition of
$$C\frac{f(a, b)(a + b - 2)}{(1 - 2Cf(a, b))} < 1 \hspace{0.1in} \text{given that} \hspace{0.1in} 1 - 2Cf(a, b) \neq 0.$$
We shall assume the latter condition that $1 - 2Cf(a, b) \neq 0$. Since the two supports of the bump functions don't intersect, $f(a, b) < \frac{1}{2}$. This translates to $1 - 2Cf(a, b) > 0$ where we note that $0 < C < 1$. We obtain
$$Cf(a, b)(a + b - 2) < 1 - 2Cf(a, b) \iff Cf(a, b)(a + b) < 1 \iff f(a, b) < \frac{1}{C(a + b)}.$$
Thus, choosing $f(a, b) = \frac{1}{C(a + b) + 2}$, we obtain such an $f$ and thus such a $g$ from (2). Notice that if $a = b = 1$, then $g \equiv 0$ and $\beta \equiv 1$. Since $\beta \equiv a$ near $0$ and $\beta \equiv b$ near $1$, we may just set $\beta$ to equal $a$ at $(-\infty, 0]$ and $\beta$ to equal $b$ at $[1, \infty)$. Since $f(a, b)$ and $g(a, b)$ are both smooth, this construction varies smoothly on $a$ and $b$. This proves the claim.
\end{proof}

\begin{proposition}\thlabel{Canonical2}
There exists a smooth function $\beta \in C^\infty(\mathbb{R}^+ \times \mathbb{R}^+ \times \mathbb{R}^+ \times \mathbb{R}; \mathbb{R}^+)$ such that for all $a, b, c \in \mathbb{R}^+$, $\beta(a, b, c, x)$ equals $a$ for $x$ near $0$, equals $b$ for $x$ near $1$, equals $c$ for $x$ near $-1$, $\beta(1, 1, 1, x) \equiv 1$, and
$$\int_{-1}^0 \beta(a, b, c, x) dx = \int_0^1 \beta(a, b, c, x) dx = 1.$$
\end{proposition}
\begin{proof}
Take the $\beta_1$ constructed in \thref{Potato} and construct $\beta$ as $\beta_1(a, b, x)$ for $x \ge 0$ and $\beta_1(c, a, -x)$ for $x \le 0$. Then since $\beta_1(a, b, x) \equiv \beta_1(c, a, -x)$ near $x = 0$, it follows that $\beta$ is smooth. Then the properties of $\beta_1$ carry through so it follows that $\beta$ is the desired function.
\end{proof}

\begin{lemma}\thlabel{Dumb}
Let $U$ be a connected bounded open set in $\mathbb{R}^n$ with smooth connected boundary, $B$ a large open ball of radius $r$ centered at some point $x_0$ that contains $\overline{U}$. Then there exists a finite cover $(U_j)_{j = 0}^m$ of $B$ such that $U_j$ is diffeomorphic to $(0, 1)^n$, each $U_j$ intersects $\partial U$, and $U_j \cap \partial U$ is diffeomorphic to $(0, 1)^{n - 1}$. Furthermore, the covering $U_i$ can be chosen such that for $B$ with sufficiently large radius, $U_j \cap U$ does not change, $m$ does not depend on the radius of $B$, and there exists $C > 0$ such that
$$\frac{\mu(B)}{\min_{U_i \cap U_j \neq 0} \mu(U_i \cap U_j )} \le C.$$
\end{lemma}

\begin{proof}
Let $V_j$ be a finite cover of $B$ by open balls of a fixed radius such that each $V_j \cap \partial U$ is either empty or diffeomorphic to $(0, 1)^{n - 1}$. We define $U_j$ as follows: if $V_j$ intersects $\partial U$, then we define $U_j$ to be $V_j$. Otherwise, take a line segment from the center of $V_j$ that intersects $\partial U$ once at a point $q_j$. Let $V_{j'}$ cover $q_j$. If $V_j$ and $V_{j'}$ intersect, then we define $U_j$ to be $V_j \cup V_{j'}$. Otherwise, we define $U_j$ to be a small tubular neighborhood of the line segment contained in all the open balls the line segment intersects together with $V_j$ and $V_{j'}$. This covering covers all of $B$ since $V_j$ covers all of $B$ and since $U_i$ is composed of either two balls of the same radius or a hypercube and two balls of the same radius, it is contractible and diffeomorphic to $(0, 1)^n$. Furthermore, by contruction $U_j \cap \partial U$ is diffeomorphic to $(0, 1)^{n - 1}$.\\\\
Let
$$p = \frac{\mu(B)}{\min_{U_i \cap U_j \neq \emptyset} \mu(U_i \cap U_j)}.$$
We show that for sufficiently large radius of $B$, the cover $U_i$ can be chosen such that $p \lesssim 1$ and $m$ is fixed. Suppose $B'$ is an open ball with radius $r'$ that contains $B$ with center $x_0$, $B''$ an open ball with center $x_0$ and with radius $r'' < r$. We may elongate the portion of $U_i$ that lies outside $U$ so that it intersects the boundary of $B$. This can be done by taking the union of $U_i$ with a small tubular neighborhood around a simple path from a point in $U_i \cap (B \setminus U)$ to $\partial B$ that does not intersect $U$. Let $(O_i)_{i = 1}^N$ be open sectors of $B'$ such that $\mu(O_i)$ is nearly $\mu(B')$ and such that they cover $B' \setminus \{x_0\}$. Take $O_i' = O_i \cap B' \setminus B''$. Each $U_j$ touches some $O_i'$ so we may redefine $U_j$ to be $U_j \cup O_i'$. Thus, $m$ does not depend on the radius of $U$. Let $r$ denote the radius of $B'$ Noticing that $\mu(U_i \cap U_j) \sim r'^n - r^n$, and $\mu(B') \sim r^n$, we do have $p \lesssim 1$ for sufficiently large $r'$. Since the redefined $U_i$ just has an additional annular sector attached to it or just a tubular neighborhood of a simple path, it is diffeomorphic to $(0, 1)^{n}$, and since the intersection $U_i \cap \partial U$ does not change, $U_i \cap \partial U$ is still diffeoorphic to $(0, 1)^{n - 1}$. This proves the lemma.  
\end{proof}

\end{document}